\documentclass[10pt, letterpaper,reqno]{amsart}
\usepackage{amssymb}
\usepackage{amsmath}
\usepackage{amsfonts, color}
\usepackage{graphicx}
\usepackage{mathrsfs,amssymb, slashed, cite}

\usepackage{hyperref}
\usepackage{cleveref}

\newcommand{\R}{\mathbb{R}}

\newcommand{\F}{\mathcal{F}}

\renewcommand{\H}{\mathcal{H}}

\newtheorem{lemma}{Lemma}[section]

\newtheorem{theorem}{Theorem}[section]
\numberwithin{equation}{section}

\newcommand{\eps}{\varepsilon}

\newcommand{\qtq}[1]{\quad\text{#1}\quad}

\begin{document}

\title[Stability]{Stability of nonlinear recovery from \\ scattering and modified scattering maps}

\author[G. Chen]{Gong Chen}
\address{Department of Mathematics, Georgia Institute of Technology}
\email{gc@math.gatech.edu}

\author[J. Murphy]{Jason Murphy} 
\address{Department of Mathematics, University of Oregon}
\email{jamu@uoregon.edu}

\maketitle

\begin{abstract} We prove stability estimates for the recovery of the nonlinearity from the scattering or modified scattering map for one-dimensional nonlinear Schr\"odinger equations.  We consider nonlinearities of the form $a(x) |u|^p u$ for $p\in [2,4]$ and $[1+a(x)]|u|^2 u$, where $a$ is a localized function.  In the first case, we show that for $p\in(2,4]$ we may obtain a H\"older-type stability estimate for recovery via the scattering map, while for $p=2$ we obtain a logarithmic stability estimate.  In the second case, we show a logarithmic stability estimate for recovery via the modified scattering map. 
\end{abstract}

\section{Introduction}

We consider nonlinear Schr\"odinger equations of the form
\begin{equation}\label{nls1}
i\partial_t u = -\Delta u + a(x)|u|^p u, \quad (t,x)\in\R\times\R
\end{equation}
and
\begin{equation}\label{nls2}
i\partial_t u = -\Delta u + [1+a(x)]|u|^2 u,\quad (t,x)\in \R\times\R.
\end{equation}
We take $p\geq 2$ and let $a$ be a localized inhomogeneity.  Our interest in this paper is the problem of recovering the coefficient $a(x)$ from the small-data scattering (or modified scattering) map.  More precisely, we seek stability estimates that control the difference between the inhomogeneities by the difference in the scattering (or modified scattering) maps.  

In what follows, given a Hilbert space $\H$ we write $B_\eta\subset \H$ to denote the ball $B_\eta=\{f\in \H:\|f\|_{\H}<\eta\}$, where $\eta$ is always understood to be a small parameter.  We write $e^{it\Delta}$ for the free Schr\"odinger propagator, defined as the Fourier multiplier operator with symbol $e^{-it\xi^2}$. We use $H^{1,1}$ to denote the weighted Sobolev space, consisting of $f\in H^1$ for which $xf\in L^2$.

\bigskip

\textbf{Scattering map.} We begin by discussing the scattering map for \eqref{nls1}.  We first state what is meant precisely by small-data scattering.

\begin{theorem}[Small-data scattering]\label{T:scatter1} Let $2\leq p\leq 4$ and $a\in L^{\frac{2}{4-p}}(\R)$. There exists $0<\eta\ll 1$ such that for any $u_0\in B_\eta\subset L^2(\R)$, there exists a unique global solution $u$ to \eqref{nls1} and $u_+\in L^2(\R)$ such that
\begin{equation}\label{scattering}
u|_{t=0}=u_0 \qtq{and} \lim_{t\to\infty} \|u(t)-e^{it\Delta} u_+\|_{L^2(\R)} = 0.
\end{equation}
\end{theorem}
This result is fairly standard; we sketch the proof in Section~\ref{S:scatter} below.  Using this theorem, we may define the \emph{scattering map} $S_a:B_\eta\subset L^2\to L^2$ by
\[
S_a(u_0)=u_+. 
\]
By now, it is well-known that the map $a\mapsto S_a$ is injective (see e.g. \cite{Strauss, Murphy, Weder1, Weder2, KMV1, KMV2}). Our first theorem is a quantitative version of this fact in the form of a H\"older-type stability estimate, in the spirit of our previous work \cite{ChenMurphy2}.

\begin{theorem}[Stability for \eqref{nls1}, $p>2$]\label{T1} Let $2< p\leq 4$, $a,b\in L^{\frac{2}{4-p}}\cap W^{1,\infty}$, and let $0<\eta\ll1$ be sufficiently small.  Let $S_a,S_b:B_\eta\subset L^2\to L^2$ denote the scattering maps for \eqref{nls1} with nonlinearities $a|u|^p u$, $b|u|^p u$, respectively, and define
\begin{equation}\label{Lip}
\|S_a-S_b\| := \sup\biggl\{\frac{\|S_a(\varphi)-S_b(\varphi)\|_{L^2}}{\|\varphi\|_{L^2}}: \varphi\in B_\eta\backslash\{0\}\biggr\}. 
\end{equation}
There exists $\theta=\theta(p)\in(0,\frac{p-2}{3p-2})$ such that if $0\leq \|S_a-S_b\|\ll 1$, then
\[
\|a-b\|_{L^\infty} \lesssim_p \|S_a-S_b\|^{\theta},
\]
with implicit constant depending on the $L^{\frac{2}{4-p}}$ and $W^{1,\infty}$ norms of $a$ and $b$. 
\end{theorem}

In fact, this result can be proven by a simple adaptation of the arguments appearing in our previous work \cite{ChenMurphy2}.  We include it here primarily to demonstrate that the exponent $\theta(p)$ that we can obtain degenerates to zero as $p\to 2$.  In particular, H\"older stability seems to break down for $p=2$ (although our techniques do not readily provide any counterexample). 

As a replacement, we show that at the $p=2$ endpoint we are able to obtain a \emph{logarithmic} stability estimate.  In particular, we will prove the following result.

\begin{theorem}[Stability for \eqref{nls1}, $p=2$]\label{T2} Let $a,b\in L^1\cap H^1$. Let $S_a,S_b:B_\eta\subset L^2\to L^2$ denote the scattering maps for \eqref{nls1} with nonlinearities $a|u|^2 u$, $b|u|^2 u$, respectively, and define $\|S_a-S_b\|$ as in \eqref{Lip}. If $0\leq \|S_a-S_b\|\ll 1$, then
\[
\|a-b\|_{L^\infty} \lesssim \biggl|\log\biggl(\frac{\|S_a-S_b\|}{1+\|a\|_{L^1\cap H^1}^2+\|b\|_{L^1\cap H^1}^2}\biggr)\biggr|^{-1}
\]
with implicit constant depending on the $L^1$ and $H^1$ norms of $a$ and $b$. 
\end{theorem}

The strategy of proof for both Theorem~\ref{T1} and Theorem~\ref{T2} is to apply the Born approximation to the scattering map (approximating the full nonlinear solution by its first Picard iterate, cf. Lemma~\ref{L:Born1} below) and then to specialize to the case of rescaled Gaussian data.  For Theorem~\ref{T1}, one then observes that (to leading order) the scattering map evaluates the convolution of the inhomogeneity with a family of approximate identities (cf. Lemma~\ref{L:approx1} below), allowing us to recover the inhomogeneity in a quantitative fashion.  

The argument breaks down precisely when $p=2$, as in this case the convolution kernels arising in the argument fail to belong to $L^1$.  Nonetheless, by direct calculation we are still able to extract the leading order behavior in this case (see Lemma~\ref{L:approx2} below), which provides a suitable replacement for the approximate identity lemma and allows us to recover the coefficient.  As we will see, the divergence is logarithmic, which then leads directly to the logarithmic-type stability estimate appearing in Theorem~\ref{T2}. In particular, Lemma~\ref{L:approx2} is really the key new technical ingredient in this work. 

\bigskip

%%%
\textbf{Modified scattering map.} We turn our attention now to the model \eqref{nls2}, which we view as a `short-range' perturbation of the standard `long-range' cubic NLS.  In this setting, scattering in the sense of \eqref{scattering} cannot occur unless the solution vanishes identically (see e.g. \cite{Strauss, Glassey1, Glassey2, Barab}).  Instead, one must incorporate a logarithmic phase correction to correctly describe the long-time asymptotic behavior.  We refer the reader to \cite{CazenaveNaumkin, DeiftZhou, HayashiNaumkin, IfrimTataru, KatoPusateri, LindbladSoffer, MurphyRIMS} for results of this type for the standard power-type NLS.  

In \cite{ChenMurphy1}, we firstly demonstrated that modified scattering persists in the presence of a localized perturbation.  In the following statement, $\F$ denotes the Fourier transform. 

\begin{theorem}[Small-data modified scattering, \cite{ChenMurphy1}]\label{T:scatter2} Suppose $a\in L^1\cap L^\infty$ with $xa\in L^2$ and $\partial_x a\in L^1$. If $\|u_0\|_{H^{1,1}}$ is sufficiently small, then there exists a unique global solution $u$ to \eqref{nls2} and $w_+\in L^\infty$ such that $u|_{t=0}=u_0$ and 
\[
\lim_{t\to\infty} \biggl\| \exp\biggl\{i\int_0^t |\F e^{-is\Delta}u(s)|^2\tfrac{ds}{2s+1}\biggr\}\F e^{-it\Delta}u(t)-w_+\biggr\|_{L^\infty} = 0.
\]
\end{theorem}

Using Theorem~\ref{T:scatter2}, we define the \emph{modified scattering map} $\tilde S_a:B_\eta\subset H^{1,1}\to L^\infty$ by
\[
\tilde S_a(u_0)=w_+.
\]
The main result in \cite{ChenMurphy1} showed that, as in the case of the unmodified scattering, the map $a\mapsto \tilde S_a$ is injective.  By extending the arguments used to obtain Theorem~\ref{T2} (taking into account the structure of the modified scattering map), we are able to obtain a logarithmic stability estimate in this setting as well. 

\begin{theorem}[Stability for \eqref{nls2}]\label{T3} Let $a,b$ satisfy the hypotheses of Theorem~\ref{T:scatter2}. Assume additionally that $\partial_xa,\partial_xb\in L^2$. Let $\tilde S_a,\tilde S_b:B_\eta\subset H^{1,1}\to L^\infty$ denote the modified scattering maps for \eqref{nls2} with inhomogeneities $a,b$, respectively, and define
\begin{equation}\label{Lip2}
\|\tilde S_a-\tilde S_b\|:=\sup\biggl\{\frac{\|\tilde S_a(\varphi)-\tilde S_b(\varphi)\|_{L^\infty}}{\|\varphi\|_{H^{1,1}}}:\varphi\in B_\eta\backslash\{0\}\biggr\}. 
\end{equation}
If $0\leq \|\tilde S_a - \tilde S_b\|\ll 1$, then
\[
 \|a-b\|_{L^\infty} \lesssim \biggl| \log\biggl(\frac{\|\tilde S_a-\tilde S_b\|}{1+C_{a,b}}\biggr)\biggr|^{-1}
 \]
where $C_{a,b}$ and the implicit constant depend on the norms of $a$ and $b$. 
\end{theorem}

The rest of this paper is organized as follows:  In the remainder of the introduction, we introduce some notation and sketch the proof of small-data scattering.  In Section~\ref{S:2}, we introduce the Born approximation and describe the structure of the scattering and modified scattering maps.  In Section~\ref{S:3}, we specialize to Gaussian data and prove the approximate identity lemmas, including the key new ingredient Lemma~\ref{L:approx2}.  In Section~\ref{S:4} we prove Theorems~\ref{T1}~and~\ref{T2}.  Finally, in Section~\ref{S:5} we prove Theorem~\ref{T3}.

\subsection{Notation.} We use the standard notation $A\lesssim B$ to denote $A\leq CB$ for some $C>0$, using subscripts to denote dependence on various parameters.  We also use the `big oh' notation $\mathcal{O}$.

We write $\langle\cdot,\cdot\rangle$ for the standard $L^2$ inner product.  We use $L_t^q L_x^r$ to denote Lebesgue space-time norms.  We write $\F$ for the Fourier transform and use the standard notation $\hat\varphi = \F\varphi$. 

\subsection{Small data scattering.}\label{S:scatter} The proof of Theorem~\ref{T:scatter1} follows along standard lines (see e.g. \cite[Theorem~3.2]{Murphy}, or \cite{Cazenave} for a textbook treatment).  We sketch the proof briefly here. 

\begin{proof}[Proof of Theorem~\ref{T:scatter1}] Using a contraction mapping argument, we can construct a solution to the following Duhamel formula for \eqref{nls1}:
\begin{equation}\label{Duhamel}
u(t)=[\Phi u](t):=e^{it\Delta}u_0-i\int_0^t e^{i(t-s)\Delta}[a|u|^p u(s)]\,ds. 
\end{equation}
We subsequently show that $\{e^{-it\Delta}u(t)\}_{t\geq 0}$ is Cauchy sequence in $L^2$ as $t\to\infty$. 

Both steps are accomplished primarily by applying Strichartz estimates for the linear Schr\"odinger equation (cf. \cite{GinibreVelo, KeelTao, Strichartz}), which in particular requires the following nonlinear estimate:
\[
\| a|u|^p u\|_{L_t^{\frac{4}{3}} L_x^1} \lesssim \|a\|_{L_x^{\frac{2}{4-p}}}\|u\|_{L_t^{\frac{4(p+1)}{3}} L_x^{\frac{2(p+1)}{p-2}}}^{p+1}. 
\]
Here we observe that $L_t^{\frac{4(p+1)}{3}}L_x^{\frac{2(p+1)}{p-2}}$ is a Schr\"odinger admissible space for any $p\in[2,4]$, while $L_t^{\frac{4}{3}}L_x^1$ is a dual Strichartz space. 
\end{proof}

\subsection{Acknowledgements} G.C. was supported by NSF grant DMS-2350301 and Simons Foundation grant MP-TSM-00002258. J.M. was supported by NSF grant DMS-2350225 and Simons Foundation grant MPS-TSM-00006622.  

%%%%%%%%%%%%%%%%
\section{Structure of the scattering maps}\label{S:2}

Using the Duhamel formula \eqref{Duhamel} for \eqref{nls1} and the definition of the scattering map, we have
\[
S_a(u_0) = u_0 - i\int_0^\infty e^{-it\Delta}[a|u|^p u(s)]\,ds. 
\]

As in \cite{ChenMurphy2} (and many other related works, in fact), this yields the following approximation to the scattering map, in which the full nonlinear solution is replaced by the first Picard iterate. The proof relies on the Duhamel formula and the estimates obtained in the proof of small data scattering (Theorem~\ref{T1}).  The precise statement we need is the following. 

\begin{lemma}[Born approximation]\label{L:Born1} Let $S_a:B_\eta\subset L^2\to L^2$ be the scattering map for \eqref{nls1}, with $p\in[2,4]$.  For $\varphi\in B_\eta$, 
\[
i \langle S_a(\varphi)-\varphi,\varphi\rangle = \int_0^\infty\int_\R a(x) |e^{it\Delta}\varphi|^{p+2}\,dx\,dt + \mathcal{O}\bigl\{\|a\|_{L^{\frac{2}{4-p}}}^2\|\varphi\|_{L^2}^{2(p+1)}\bigr\}
\]
\end{lemma}

In \cite{ChenMurphy1}, we used the Born approximation to obtain the following decomposition for the \emph{modified} scattering map $\tilde S_a$.

\begin{lemma}[Structure of the modified scattering map]\label{L:Born2} Let $a$ satisfy the hypotheses of Theorem~\ref{T:scatter2} and let $\tilde S_a:B_\eta\subset H^{1,1}\to L^\infty$ be the modified scattering map for \eqref{nls2}. Let $\varphi\in\mathcal{S}$. For $\eps>0$ sufficiently small, we have
\begin{equation}\label{structure}
\begin{aligned}
\langle \tilde S_a(\eps\varphi),\hat\varphi\rangle & = \eps\langle\hat\varphi,\hat\varphi\rangle + \tfrac{1}{2i}\log(1+\tfrac{1}{2\eps})\langle |\tilde S_a(\eps\varphi)|^2\tilde S_a(\eps\varphi),\hat\varphi\rangle + \eps^3 \mathcal{Q}_\eps[\varphi] \\
& \quad - i\eps^3\int_0^\infty\int_\R a(x)|e^{it\Delta}\varphi(x)|^4\,dx\,dt + \mathcal{O}(\eps^4\|\varphi\|_{H^{1,1}}^5),
\end{aligned}
\end{equation}
where
\[
\mathcal{Q}_\eps[\varphi]:=\int_\eps^\infty \tfrac{1}{2\it}\iiint [e^{-i\frac{\eta\sigma}{2t}}-1]\varphi(z-\eta)\varphi(z-\sigma)\bar\varphi(z)\bar\varphi(z-\eta-\sigma)\,dz\,d\eta\,d\sigma\,dt. 
\]
\end{lemma}

\begin{proof} This lemma appears as Proposition~4.1 in \cite{ChenMurphy1}.  The proof relies on a careful examination of the estimates used in obtaining the modified scattering result (i.e. Theorem~\ref{T:scatter2} in the present work).  In particular, one approximates the full solution $u(t)$ by $e^{it\Delta}[\eps\varphi]$ (the Born approximation) and then separates the cubic terms from the higher order terms, that is, the terms containing at least four copies of the solution/data. 

In the work \cite{ChenMurphy1}, the higher order terms are estimated simply as $\mathcal{O}(\eps^4)$.  In fact, the proof shows that all of the terms can ultimately be estimated in terms of the $H^{1,1}$-norm of the data $\varphi$, which (in the setting of \cite{ChenMurphy1}) is a fixed Schwartz function.  In our setting, we will consider $\varphi$ from a family of rescaled Gaussians with growing $\dot H^{1}$-norm; thus we track the dependence on this norm in our analysis. The implicit constant in the error term will also depend on various norms of $a$ appearing in the estimates, which we do not track in this work. 
\end{proof}

%%%
\section{Approximation to the Identity}\label{S:3}

The basic idea in recovering the nonlinear coefficient $a(x)$ from the scattering map is to take data corresponding to a family of rescaled Gaussians concentrating to a fixed point $x_0$.  Applying the Born approximation Lemma~\ref{L:Born1}, one finds that to leading order, the scattering map (minus the identity) evaluates the convolution of $a$ with a family of approximate identities $K_\sigma$ at $x_0$ (up to a computable constant), at least when the power $p$ exceeds $2$.  The argument breaks down at $p=2$, as in this case the convolution kernels fail to belong to $L^1$.  However, as we will see, the failure is only logarithmic and we can still extract the leading order behavior in order to recover the nonlinear coefficient.

We begin with the simpler case $p>2$. The following lemma follows directly from \cite[Proposition~2.2]{ChenMurphy2}.  The proof is simply a quantitative version of the standard approximate identity lemma.  Our identity differs from that in \cite{ChenMurphy2} by a factor of $2$, reflecting the fact that here we integrate over $t\in[0,\infty)$ rather than $t\in\R$.

\begin{lemma}[Approximate identity, $p>2$]\label{L:approx1} Let $a\in W^{1,\infty}$.  Define $\varphi(x)=\exp\{-\frac{x^2}{4}\}$ and $\varphi_{\sigma,x_0}(x) = \varphi(\tfrac{x-x_0}{\sigma})$ for $x_0\in\R$ and $0<\sigma\ll1$.  Let $p>2$ and define
\[
\lambda(p)=\pi^{\frac32}(p+2)^{-\frac12}\frac{\Gamma(\frac{p}{4}-\frac12)}{\Gamma(\frac{p}{4})}. 
\]
Then for any $0<s<1-\frac{2}{p}$, 
\[
\biggl| \int_0^\infty \int_\R a(x) |e^{it\Delta}\varphi_{\sigma,x_0}(x)|^{p+2}\,dx\,dt - \sigma^3\lambda(p)a(x_0)\biggr| \lesssim_s \sigma^{3+s}\|a\|_{W^{1,\infty}}.
\] 
\end{lemma}

As one can readily observe, we have $|\lambda(p)|\to\infty$ as $p\downarrow2$.  The following lemma provides a substitute for Lemma~\ref{L:approx1} for the case $p=2$.  It is the essential new technical ingredient in this paper. 

\begin{lemma}[Approximate identity, $p=2$]\label{L:approx2} Let $a\in H^1\cap L^1$.  Define $\varphi(x)=\exp\{-\frac{x^2}{4}\}$ and $\varphi_{\sigma,x_0}(x) = \varphi(\tfrac{x-x_0}{\sigma})$ for $x_0\in\R$ and $0<\sigma\ll1$. Then
\[
\biggl| \int_0^\infty \int_\R a(x) |e^{it\Delta}\varphi_{\sigma,x_0}(x)|^4\,dx\,dt -\tfrac{1}{\sqrt{2}}\sigma^3|\log \sigma| a(x_0)\biggr| \lesssim \sigma^3\|a\|_{H^1}.
\]
\end{lemma}

\begin{proof} Let
\[
K(x)=\int_0^\infty |e^{it\Delta}\varphi(x)|^4\,dt = \int_0^\infty (1+t^2)^{-1}\exp\bigl\{-\tfrac{|x|^2}{1+t^2}\bigr\}\,dt, 
\]
so that by directly computing a Gaussian integral we may obtain
\begin{align*}
\hat K(\xi) = (2\pi)^{-\frac12}\int_\R e^{-ix\xi}K(x)\,dx & = \tfrac{1}{\sqrt{2}}\int_0^\infty (1+t^2)^{-\frac12} \exp\{-\tfrac{\xi^2(1+t^2)}{4}\}\,dt. 
\end{align*}
Now, by Plancherel's Theorem (and the fact that $K$ and $\hat K$ are real-valued) we may write
\begin{align*}
\int_0^\infty \int_\R a(x) |e^{it\Delta}\varphi_{\sigma,x_0}(x)|^4\,dx\,dt & = \sigma^3\int_\R e^{ix_0\xi}\hat a(\xi)\hat K(\sigma\xi)\,d\xi \\
& = \sigma^3\int_{|\sigma\xi|\leq 1}e^{ix_0\xi}\hat a(\xi)\hat K(\sigma\xi)\,d\xi \\
& \quad + \sigma^3\int_{|\sigma\xi|>1}e^{ix_0\xi}\hat a(\xi)\hat K(\sigma\xi)\,d\xi.
\end{align*}
Noting that $\|\hat K\|_{L_\xi^\infty(|\xi|\geq 1)}\lesssim 1$, we may obtain
\[
\sigma^3\biggl|\int_{|\sigma\xi|>1}e^{ix_0\xi}\hat a(\xi)\hat K(\sigma\xi)\,d\xi\biggr| \lesssim \sigma^3 \|\hat a\|_{L^1} \lesssim \sigma^3\|a\|_{H^1},
\]
which is acceptable.  To proceed, we need to understand the behavior of $\hat K(\xi)$ for $0<|\xi|\leq 1$. To this end, we use the Taylor expansion of $\exp\{-\tfrac{t^2\xi^2}{4}\}$ for $t\in[0,|\xi|^{-1}]$ to write
\begin{align}
\hat K(\xi) & =\tfrac{1}{\sqrt{2}}\exp\{-\tfrac{\xi^2}{4}\}\int_0^{|\xi|^{-1}} (1+t^2)^{-\frac12}\,dt \label{Khat1}
\\
&\quad + \tfrac{1}{\sqrt{2}} \exp\{-\tfrac{\xi^2}{4}\}\sum_{n=1}^\infty\tfrac{(-1)^n \xi^{2n}}{4^n n!} \int_0^{|\xi|^{-1}} (1+t^2)^{-\frac12}t^{2n}\,dt \label{Khat2}
\\ 
& \quad + \tfrac{1}{\sqrt{2}} \exp\{-\tfrac{\xi^2}{4}\} \int_{|\xi|^{-1}}^\infty (1+t^2)^{-\frac12} \exp\{-\tfrac{t^2\xi^2}{4}\}\,dt.\label{Khat3} 
\end{align}
By direct calculation and the identity $\sinh^{-1}(x) = \log(x+\sqrt{1+x^2})$, 
\begin{align*}
\eqref{Khat1} & = \tfrac{1}{\sqrt{2}} \exp\{-\tfrac{\xi^2}{4}\} \sinh^{-1}(|\xi|^{-1}) \\ 
& = \tfrac{1}{\sqrt{2}}\exp\{-\tfrac{\xi^2}{4}\}\bigl[\log(\tfrac{1}{|\xi|}) + \log[1+(1+\xi^2)^{\frac12}]\bigr] \\
& = \tfrac{1}{\sqrt{2}}\log(\tfrac{1}{|\xi|}) + \mathcal{O}(1) \qtq{for}|\xi|\leq 1. 
\end{align*} 
Furthermore, we have
\[
\eqref{Khat2} + \eqref{Khat3} = \mathcal{O}(1) \qtq{for}|\xi|\leq 1. 
\]
Thus we may write 
\[
\hat K(\sigma\xi) = \tfrac{1}{\sqrt{2}}\log(\tfrac{1}{\sigma}) + \tfrac{1}{\sqrt{2}} \log(\tfrac{1}{|\xi|}) + \mathcal{O}(1) \qtq{for}0<|\xi|\leq \tfrac{1}{\sigma}
\]
and apply the Fourier inversion formula to obtain  
\begin{align}
\biggl| \sigma^3 \int_{|\sigma\xi|\leq 1} & e^{ix_0\xi}\hat a(\xi) \hat K(\sigma\xi)\,d\xi -  \tfrac{1}{\sqrt{2}}\sigma^3|\log \sigma| a(x_0)\biggr| \nonumber \\
& \lesssim\sigma^3|\log\sigma| \int_{|\xi|>\sigma^{-1}} |\hat a(\xi)| \,d\xi \label{Khat-e1} \\
& \quad + \sigma^3\int |\hat a(\xi)| |\log(|\xi|)|\,d\xi + \sigma^3\int |\hat a(\xi)|\,d\xi. \label{Khat-e2}
\end{align}
Noting that
\[
\int \langle \xi\rangle^s |\hat a(\xi)|\,d\xi \lesssim \|a\|_{H^1} \qtq{for}s\in(0,\tfrac12)
\]
and $\log(|\xi|)\in L^2(\{|\xi|<1\})$, we can derive that
\[
|\eqref{Khat-e1}|+|\eqref{Khat-e2}|\lesssim \sigma^3\|a\|_{H^1},
\]
which is acceptable. 

Collecting the identities and estimates above, we complete the proof.  \end{proof}

%%%%%%%%%%%%%%%%%%%%
\section{Stability for the scattering map}\label{S:4}

In this section we prove Theorems~\ref{T1}~and~\ref{T2}.

\begin{proof}[Proof of Theorems~\ref{T1}~and~\ref{T2}] We let $p\in[2,4]$ and let $S_a$ and $S_b$ be the scattering maps associated with nonlinearities $a|u|^p u$ and $b|u|^p u$, respectively. Let $B_\eta\subset L^2$ be the common domain of $S_a$ and $S_b$. 

We let $x_0\in\R$ and $\sigma>0$ and define $\varphi_{\sigma,x_0}(x_0)=\varphi(\tfrac{x-x_0}{\sigma})$, with $\varphi(x)=\exp\{-\tfrac{x^2}{4}\}$. Note that
\begin{equation}\label{Gaussian-norm}
\|\varphi_{\sigma,x_0}\|_{L^2} \lesssim \sigma^{\frac12},
\end{equation}
so that $\varphi_{\sigma,x_0}\in B_\eta$ for all $0<\sigma\ll 1$. 

First consider the case $p>2$.  We denote
\[
\|a\|_X = \|a\|_{W^{1,\infty}} + \|a\|_{L^{\frac{2}{4-p}}}.
\]

Using \eqref{Gaussian-norm}, the Born approximation (Lemma~\ref{L:Born1}), and the approximate identity lemma (Lemma~\ref{L:approx1}), we obtain
\begin{align*}
\sigma^3 \lambda(p)[a(x_0)-b(x_0)] & = i\langle S_a(\varphi_{\sigma,x_0})-S_b(\varphi_{\sigma,x_0}),\varphi_{\sigma,x_0}\rangle \\
& \quad + \mathcal{O}\bigl\{\sigma^{3+s}[\|a\|_{X}+\|b\|_{X}]\bigr\} \\
& \quad + \mathcal{O}\bigl\{\sigma^{p+1}[\|a\|_{X}^2+\|b\|_X^2]\bigr\}
\end{align*}
for $0<s<1-\tfrac{2}{p}$.  In particular, the first error term arises from Lemma~\ref{L:approx1} and the second error term from Lemma~\ref{L:Born1}.  

Thus, recalling \eqref{Lip} and again using \eqref{Gaussian-norm}, we obtain
\[
\|a-b\|_{L^\infty} \lesssim_p \sigma^{-2} \|S_a-S_b\| + \sigma^s[\|a\|_X+\|b\|_X] + \sigma^{p-2}[\|a\|_X^2+\|b\|_X^2].
\]
Noting that $1-\tfrac{2}{p}=\tfrac{1}{p}[p-2]<p-2$ and recalling the assumption $\|S_a-S_b\|\ll 1$, we choose
\[
\sigma \sim \biggl[\frac{\|S_a-S_b\|}{1+\|a\|_X^2+\|b\|_X^2}\biggr]^{\frac{1}{2+s}}
\]
in order to obtain
\[
\|a-b\|_{L^\infty} \lesssim_p \|S_a-S_b\|^{\frac{s}{2+s}},
\]
with implicit constant depending on the $X$-norms of $a$ and $b$. This yields Theorem~\ref{T1}.  We remark that since $1-\frac{2}{p}\to 0$ as $p\downarrow 2$, this H\"older-type stability estimate breaks down as $p\downarrow 2$. 

We turn now to Theorem~\ref{T2} and set $p=2$.  In this case we write
\[
\|a\|_{X} = \|a\|_{L^1}+\|a\|_{H^1}. 
\]

In this case, we combine the Born approximation (Lemma~\ref{L:Born1}) and modified approximate identity lemma (Lemma~\ref{L:approx2}) to obtain
\begin{align*}
\tfrac{1}{\sqrt{2}}\sigma^3|\log \sigma| [a(x_0)-b(x_0)] & = i\langle S_a(\varphi_{\sigma,x_0})-S_b(\varphi_{\sigma,x_0}),\varphi_{\sigma,x_0}\rangle \\
& \quad + \mathcal{O}\bigl\{\sigma^{3}[\|a\|_{X}+\|b\|_{X}]\bigr\} \\
& \quad + \mathcal{O}\bigl\{\sigma^{3}[\|a\|_{X}^2+\|b\|_X^2]\bigr\}
\end{align*}
Thus 
\[
\|a-b\|_{L^\infty} \lesssim \sigma^{-2}|\log\sigma|^{-1} \|S_a-S_b\| + |\log\sigma|^{-1}\{\|a\|_X+\|a\|_X^2+\|b\|_X+\|b\|_X^2\}. 
\]
Choosing
\[
\sigma \sim \biggl[\frac{\|S_a-S_b\|}{1+\|a\|_X^2+\|b\|_X^2}\biggr]^{\frac12}
\]
leads to the estimate
\[
\|a-b\|_{L^\infty} \lesssim \biggl|\log\biggl(\frac{\|S_a-S_b\|}{1+\|a\|_{L^1\cap H^1}^2+\|b\|_{L^1\cap H^1}^2}\biggr)\biggr|^{-1},
\]
with implicit constant depending on the $X$-norms of $a$ and $b$. 
\end{proof} 

%%%%%%%%%%%%%%%%%%%%%%%%%
\section{Stability for the modified scattering map}\label{S:5}

In this section we prove Theorem~\ref{T3}.  

\begin{proof}[Proof of Theorem~\ref{T3}] We let $\tilde S_a,\tilde S_b$ be the modified scattering maps associated to inhomogeneities $a,b$. We write $B_\eta\subset H^{1,1}$ for the common domain of $S_a$ and $S_b$.  We let $x_0\in\R$, $\sigma>0$, and $\eps>0$, and define 
\[
\varphi_{\sigma,x_0}(x_0)=\varphi(\tfrac{x-x_0}{\sigma}),\qtq{with}\varphi(x)=\exp\{-\tfrac{x^2}{4}\}.
\]
Note that
\begin{equation}\label{Gaussian-norm2}
\|\varphi_{\sigma,x_0}\|_{H^{1,1}} \lesssim \sigma^{-\frac12},
\end{equation}
so that $\eps\varphi_{\sigma,x_0}\in B_\eta$ for $\eps\ll\sigma^{\frac12}$.  We also note that $\|\hat\varphi_{\sigma,x_0}\|_{L^1}\equiv \|\hat\varphi\|_{L^1}\lesssim 1$. 

Using Lemma~\ref{L:Born2}, we may write
\begin{align}
-i\eps^3\int_0^\infty&\int_\R [a(x)-b(x)]|e^{it\Delta}\varphi_{\sigma,x_0}|^4\,dx\,dt \label{mod1}\\
& = \langle \tilde S_a(\eps\varphi_{\sigma,x_0})-\tilde S_b(\eps\varphi_{\sigma,x_0}),\hat\varphi_{\sigma,x_0}\rangle \label{mod2}\\  
& \quad + \tfrac{1}{2i}\log(1+\tfrac{1}{2\eps})\langle [|\tilde S_a|^2\tilde S_a](\eps\varphi_{\sigma,x_0})-[|\tilde S_b|^2\tilde S_b](\eps\varphi_{\sigma,x_0}),\hat\varphi_{\sigma,x_0}\rangle \label{mod3} \\
& \quad +\mathcal{O}(\eps^4\sigma^{-\frac52}),\label{mod4}
\end{align}
where the final error term also depends on various norms of $a$ and $b$. 

Recalling the definition in \eqref{Lip2}, we begin with the estimate
\begin{align*}
|\eqref{mod2}| & \lesssim \eps \|\tilde S_a - \tilde S_b\|\, \|\varphi_{\sigma,x_0}\|_{H^{1,1}}\|\hat\varphi_{\sigma,x_0}\|_{L^1}  \\
& \lesssim \eps\sigma^{-\frac12}\|S_a-S_b\|\|\hat\varphi\|_{L^1} \\
&\lesssim \eps\sigma^{-\frac12}\|\tilde S_a-\tilde S_b\|.
\end{align*}
Using the fact that
\[
\|\tilde S_a\|_{H^{1,1}\to L^\infty}+\|\tilde S_b\|_{H^{1,1}\to L^\infty} \lesssim 1, 
\]
we estimate
\begin{align*}
|\eqref{mod3}| & \lesssim \log(1+\tfrac{1}{2\eps})\|\hat\varphi_{\sigma,x_0}\|_{L^1}\\
& \quad \times \{\|\tilde S_a(\eps\varphi_{\sigma,x_0})\|_{L^\infty}^2+\|\tilde S_b(\eps\varphi_{\sigma,x_0})\|_{L^\infty}^2\}\|[\tilde S_a-\tilde S_b](\eps\varphi_{\sigma,x_0})\|_{L^\infty} \\
& \lesssim \eps^3\log(1+\tfrac{1}{2\eps})\|\hat\varphi\|_{L^1}\|\varphi_{\sigma,x_0}\|_{H^{1,1}}^3\|\tilde S_a-\tilde S_b\|\\
& \lesssim \eps^3\log(1+\tfrac{1}{2\eps})\sigma^{-\frac32}\|\tilde S_a-\tilde S_b\|. 
\end{align*}

For the main term \eqref{mod1}, we apply Lemma~\ref{L:approx2} to write
\[
\eqref{mod1} = -\eps^3\sigma^3\tfrac{i}{\sqrt{2}}|\log\sigma| [a(x_0)-b(x_0)]+ \mathcal{O}(\sigma^3\eps^3).
\]

Combining the estimates above, we find
\begin{align*}
\|a-b\|_{L^\infty} & \lesssim [\eps^{-2}\sigma^{-\frac72}+\log(1+\tfrac{1}{2\eps})\sigma^{-\frac92}]|\log\sigma|^{-1}\|\tilde S_a-\tilde S_b\|  \\
& \quad + C_{a,b}[1+\eps\sigma^{-\frac{11}{2}}]|\log\sigma|^{-1},
\end{align*}
where $C_{a,b}$ is a constant depending on various norms of $a$ and $b$.  Specializing to $\eps=\sigma^{\frac{11}{2}}\ll\sigma^{\frac12}$ and using the straightforward estimate $|\log\langle x\rangle|\lesssim_{\beta} \langle x\rangle^{\beta}$ for any $\beta>0$, we can simplify the estimate above to
\[
\|a-b\|_{L^\infty} \lesssim [\sigma^{-\frac{29}{2}}\|\tilde S_a - \tilde S_b\| + C_{a,b}]|\log\sigma|^{-1} 
\]
We now choose 
\[
\sigma \sim \biggl[\frac{\|\tilde S_a-\tilde S_b\|}{1+C_{a,b}}\biggr]^{\frac{2}{29}}
\]
to obtain the desired estimate
\[
\|a-b\|_{L^\infty} \lesssim \biggl| \log\biggl(\frac{\|\tilde S_a-\tilde S_b\|}{1+C_{a,b}}\biggr)\biggr|^{-1},
\]
with implicit constant depending on norms of $a$ and $b$. \end{proof}

\end{document}